\documentclass[12pt]{article}
\usepackage[utf8]{inputenc}

\usepackage{amsmath}
\usepackage{mathtools}
\usepackage{amssymb}
\usepackage{amsthm}
\usepackage{graphicx}
\usepackage{amscd}
\usepackage{epic, eepic}
\usepackage{url}
\usepackage{color}
\usepackage[utf8]{inputenc} 
\usepackage{makebox}
\usepackage{comment}
\usepackage{enumerate}   
\usepackage{hyperref}
\usepackage[]{todonotes}
\theoremstyle{plain}
\newtheorem{theorem}{\bf Theorem}[section]
\newtheorem{lemma}[theorem]{\bf Lemma}

\newtheorem{prop}[theorem]{\bf Proposition}

\newtheorem{construction}[theorem]{\bf Construction}
\newtheorem{nota}[theorem]{\bf Notation}

\newtheorem{remark}[theorem]{\bf Remark}
\newtheorem{defi}[theorem]{\bf Definition}

\def\ex{\hbox{\rm ex}}

\title{Multicolor Turán numbers II - a generalization of the Ruzsa-Szemerédi theorem and new results on cliques and odd cycles }

\author{Benedek Kovács\thanks{ELTE Linear Hypergraphs  Research Group, Eötvös Loránd University, Budapest, Hungary. The author is partially supported by  the ÚNKP, New National Excellence Program of the Ministry for Innovation and
		Technology from the source of the National Research, Development and Innovation Fund.
		E-mail: {\tt benoke98@student.elte.hu}}%\todo{affiliáció/mail/támogató}}
	\and Zolt\'an L\'or\'ant Nagy\thanks{ELTE Linear Hypergraphs  Research Group,
		E\"otv\"os Lor\'and University, Budapest, Hungary. The author is supported by the Hungarian Research Grant (NKFI) No. PD  134953.  	E-mail: {\tt nagyzoli@cs.elte.hu}}
}
\date{}

\begin{document}
	\maketitle

	\begin{abstract} In this paper we continue the study of a natural generalization of  Turán's forbidden subgraph problem and the Ruzsa-Szemerédi problem.
		Let $\ex_F(n,G)$ denote the maximum number of edge-disjoint copies of a fixed simple graph $F$ that can be placed on an $n$-vertex ground set without forming  a subgraph $G$ whose edges are from different $F$-copies. The case when both $F$ and $G$ are triangles essentially gives back the theorem of Ruzsa and Szemerédi. We extend their results to the case when $F$ and $G$ are arbitrary cliques by applying a number theoretic result due to Erdős, Frankl and Rödl.  This extension in turn decides the order of magnitude for a large family of graph pairs, which will be subquadratic, but almost quadratic. Since the linear $r$-uniform hypergraph Turán problems to determine $\ex_r^{lin}(n,G)$ form a class of the multicolor Turán problem, following the identity $\ex_r^{lin}(n,G)=\ex_{K_r}(n,G)$, our results determine the linear hypergraph Turán numbers of every graph of girth $3$ and for every $r$ up to a subpolynomial factor.\\
		Furthermore, when $G$ is a triangle, we settle the case $F=C_5$ and give bounds for the cases $F=C_{2k+1}$, $k\ge 3$ as well.%\todo{ plusz: a semi-random dolgokat hangsúlyozni. Reglemma szerepe}
		
	\end{abstract}

	\section{Introduction}

	The theorem of Turán \cite{Turan} and in general, his forbidden subgraph problems initiated a broad area in extremal graph theory. Recently the second author together with Imolay, Karl and Váli 
	started to investigate a new type of  generalization of Turán-type problems  \cite{IKNV}. For a fixed pair $(F,G)$ of nonempty graphs, the aim is to determine the maximum number of edge-disjoint copies of $F$ on $n$ vertices which do not form a subgraph $G$ such that all the edges of this subgraph $G$ are from different $F$-copies. Coloring each $F$-copy with a color of its own, the forbidden configuration is a multicolor $G$. This maximum is denoted by $\ex_F(n,G)$ and we refer to it as the {\em $F$-multicolor Turán number of $G$}.
	
	Note that in the case when $F$ is just an edge (that is, $F=K_2$), we get back the original Turán problem concerning the forbidden subgraph $G$.
	
	The case $F=G=K_3$ is of considerable interest as it turns out that there is a correspondence between the famous Ruzsa-Szemerédi $(6,3)$-problem and this case of the multicolor Turán problem. 
	The (6,3)-problem asks for the maximum number of hyperedges in a 3-uniform hypergraph on $n$ vertices such that no six vertices span three or more hyperedges. By considering each hyperedge as a distinct copy of $K_3$, the two problems can be seen to be essentially the same. In 1978 Ruzsa and Szemerédi proved \cite{Ruzsa} their theorem, which reads as follows in our terminology:
	
	$$n^2e^{-O(\sqrt{\log n})}\le \ex_{K_3}(n, K_3)=o({n^2}).$$
	
	%Our problem is closely related to the well studied linear $r$-uniform hypergraph Turán problems on determining $\ex_r^{lin}(n,G)$. As it was pointed out in \cite{IKNV}, $\ex_r^{lin}(n,G)=\ex_{K_r}(n,G)$ holds, thus results for the case when $F$ is a clique are of special interest.
	
	Their lower bound uses a construction based on Behrend's result \cite{Behrend} that shows  it is possible to choose a subset $H$ of $\{1,2,\dots,n\}$ having size $ne^{-O(\sqrt{\log n})}$ which does not contain an arithmetic progression of length 3. 
	
	\bigskip
	%\todo{összefoglalni az előző cikk fő eredményeit}
	
	The second author together with Imolay, Karl and Váli \cite{IKNV} characterized those graph pairs $(F,G)$ for which the multicolor Turán number has quadratic order of magnitude, moreover they proved an asymptotic result on $\ex_F(n,G)$ in the case when the chromatic number of $F$ is less than that of $G$. %\todo{Elaboration on Edrős-Stone}
	These results provide a generalization of the famous Erdős-Stone-Simonovits theorem.
	
	\begin{theorem}[Imolay, Karl, Nagy and Váli \cite{IKNV}]\label{mainTHM1}
		$\ex_F(n,G)=\Theta(n^2)$ if and only if there is no homomorphism from $G$ to $F$, and $\ex_F(n,G)=o(n^2)$ otherwise.
	\end{theorem}

	\begin{theorem} [Imolay, Karl, Nagy and Váli \cite{IKNV}]\label{mainth2}
		If $\chi(F)<\chi(G)$, then
		$$\ex_F(n,G)\cdot |E(F)| \sim {\ex(n,G)} \sim \left({1-\frac{1}{\chi(G)-1}}\right)\binom{n}{2}.$$
	\end{theorem}
	
	Our problem is closely related to the well-studied linear $r$-uniform hypergraph Turán problems to determine $\ex_r^{lin}(n,G)$ for a graph $G$.  A hypergraph $H$ is called a Berge copy of $G$ or Berge-$G$ if we
	can choose a two-element subset of each hyperedge of $H$ to obtain a copy of $G$. A hypergraph is linear if every pair of vertices appears together in at most one hyperedge.
	$\ex_r^{lin}(n,G)$ denotes the maximum number of hyperedges in a linear $r$-uniform hypergraph on $n$ vertices which does not contain a Berge copy of $G$.  We refer to \cite{Gyori1, Furedi-Ozkahya, Gao, Gao2, Tait} for results on linear $r$-uniform hypergraph Turán problems.
	As it was pointed out in \cite{IKNV}, $\ex_r^{lin}(n,G)=\ex_{K_r}(n,G)$ holds, thus results in the case when $F$ is a clique are of special interest. 
	
	Other related variants are the generalised Turán problems, whose investigation was initiated by Alon and Shikhelman \cite{AS}. $\ex{(n, H, G)}$ denotes the maximum number of copies of $H$ in a $G$-free graph on $n$ vertices. Note that some special cases have attracted interest earlier as well, particularly the cases when $H$ and $G$ are both cliques or both cycles.  The case $H=C_5$, $G=C_3$ was the well-known pentagonal conjecture of Erdős \cite{Erdos84}, resolved via the flag algebra method \cite{Grzesik, Hatami}.
	
	For cliques, Theorem \ref{mainTHM1} and \ref{mainth2} cover the case when $F$ is a smaller clique than $G$ and shows that the function is sub-quadratic otherwise.
	Our first contribution  states  that for any fixed  $t\ge 3$%\todo{ez így kicsit fura, hogy arb. large-ot és legalább 3-at is mondunk}
	, it is possible to pack almost quadratically many copies of $K_t$ into an $n$-vertex ground set without creating a multicolor $K_3$.

	\begin{theorem}\label{K_t K_3}
		For every integer $t\ge 3$, we have $$ n^2e^{-O(\sqrt{\log n})}\le \ex_{K_t}(n,K_3)=o(n^2).$$
	\end{theorem}
	
	Note that the multicolor Turán function satisfies the following monotonicity properties (see Proposition 2.1 in \cite{IKNV}):
	
	\begin{enumerate}[(i)]
		\item if $F_1$ is a subgraph of $F_2$ and $G$ is any graph then $\ex_{F_2}(n,G) \leq \ex_{F_1}(n,G)$,
		\item if $G_1$ is a subgraph of $G_2$ and $F$ is any graph then $\ex_F(n,G_1)\leq \ex_F(n,G_2)$.
	\end{enumerate}
	
	Using property (i), we get an even more general theorem.
	
	\begin{theorem}\label{F K_3}
		Suppose that $F$ is a graph of girth $3$ (that is, $F$ contains a triangle). Then $$ n^2e^{-O(\sqrt{\log n})}\le \ex_{F}(n,K_3)=o(n^2).$$
	\end{theorem}
	
	On the other hand, (ii) together with Theorem \ref{K_t K_3} resolves all remaining cases when $F$ and $G$ are cliques, and yields
	
	\begin{theorem}\label{K_t K_r}
		For every $t\ge r\ge 3$, we have $$ n^2e^{-O(\sqrt{\log n})}\le \ex_{K_t}(n,K_r)=o(n^2)$$
		and in general, if $G$ has girth $3$ and there is a homomorphism from $G$ to $F$, then 
		$$ n^2e^{-O(\sqrt{\log n})}\le \ex_{F}(n,G) = o(n^2).$$
	\end{theorem}
	
	%\begin{theorem}
	% For every $t\ge r\ge 3$, we have $$ n^2e^{-O(\sqrt{\log n})}\le \ex_{K_t}(n,K_r)=o(n^2)$$
	%and in general, for graphs $F$ and $G$, 
	%\begin{enumerate}[(a)]
	%\item if $G$ has girth $3$, then $ n^2e^{-O(\sqrt{\log n})}\le \ex_{F}(n,G)$.
	%\item if there is a homomorphism from $G$ to $F$, then $\ex_F(n,G)=o(n^2)$.
	%\end{enumerate}
	%\end{theorem}
	
	Note that when there is no homomorphism from $G$ to $F$, we already have  $\ex_{F}(n,G) = \Theta(n^2)$ due to Theorem \ref{mainTHM1}.
	
	\begin{proof} Choose an integer $t$ such that $F\subseteq K_t$. Then 
		$$ n^2e^{-O(\sqrt{\log n})}\le \ex_{K_t}(n,K_3)\leq  \ex_{F}(n,K_3)\leq \ex_{F}(n,G),$$ according to the monotonicity and Theorem \ref{K_t K_3}. As there is a homomorphism from $G$ to $F$ (which is also true in the case $F=K_t$, $G=K_r$ with $t\ge r\ge 3$), we also have $\ex_{F}(n,G)=o(n^2)$ due to Theorem \ref{mainTHM1}.
	\end{proof}
	
	These results drive attention to the cases when the multicolor Turán number is quadratic, but the asymptotics are not determined by Theorem \ref{mainth2}: for example when $G$ is a bipartite graph, or in general, $\chi(F)\ge \chi(G)$ without a homomorphism from $G$ to $F$. A notorious example in the latter case is when $F$ is a pentagon and $G$ is a triangle, which appeared in the Erdős pentagonal conjecture as well. We continue our paper by resolving this case in the multicolor Turán setting, which was mentioned as a conjecture in \cite{IKNV}.
	
	\begin{theorem}\label{penta}
		$$\ex_{C_5}(n,K_3)=(1+o(1))\frac{n^2}{25}.$$
	\end{theorem}
	
	This theorem shows the existence of graph pairs $(F, G)$ where $\ex_F{(n,G)}\sim \frac{1}{v(F)^2}n^2$ holds, verifying the conjecture stated in \cite{IKNV}. Observe that assuming $F$ and $G$ are graphs such that there is no homomorphism from $G$ to $F$,
	$$\ex_F(n,G) \geq \frac{1}{v(F)^2}n^2-o(n^2)$$ holds by \cite[Theorem~3.1]{IKNV}. The main tool behind this result was a theorem due to Haxell and Rödl \cite{packing}, which will be used in this paper as well, so we discuss it in detail in  Section 3. %follows from an asymptotic decomposition of the blow-up graph $F[m]$ of $F$ to graphs $F$. Indeed, the existence of such a packing follows from the result of Haxell and Rödl \cite{packing} on the connection between  which we will discuss in Section 3, while $F[m]$ does not contain $G$ as a subgraph according to our assumption.
	\bigskip
	
	We discuss the case when $F$ is a longer odd cycle as well. 
	
	\begin{theorem}\label{oddlarge} If $k\ge 3$ then
		$$\ex_{C_{2k+1}}(n,K_3)\cdot |E(C_{2k+1})|>(1+o(1))\frac{n^2}{5}, $$ and
		$$\lim_{k\to \infty}\ex_{C_{2k+1}}(n,K_3)\cdot |E(C_{2k+1})|=(1+o(1))\frac{n^2}{4}. $$ 
	\end{theorem}
	
	Very recently, Balogh, Liebenau, Mattos and Morrison proved Theorem \ref{penta} in a stronger form \cite{Mattos} and their upper bound for the multicolor Turán number $\ex_{C_5}(n,K_3)$ has a better linear term than the one implicitly obtained from our proof. Their results also capture the structure of all extremal structures concerning $\ex_{C_5}(n,K_3)$.
	
	Several variants of Turán-type problems have been investigated before, here we refer to the introduction of \cite{IKNV}  for a short survey of related results. Besides, let us point out a very recent variant investigated by Gerbner \cite{Gerbner}. He assigned colors to distinct copies of a fixed subgraph $F$, while in our work, colors are essentially assigned to edges, and edges of the same color determine the copy of $F$. This implies that his results correspond to  hypergraph Turán problems concerning Berge copies of a fixed graph while our concept corresponds to \textit{linear} hypergraph Turán problems, when $F$ is a clique.
	
	%In some cases, the extremal construction for a multicolor Turán function $\ex_F(n,G)$ does not contain subgraphs isomorphic to $G$ at all \cite{AS}.
	
	The paper is organised as follows. In Section 2 we discuss a result of Alon and Shapira, which we will use for the lower bound in Theorem \ref{K_t K_3}.
	While the
	lower bound for $\ex_{K_3}(n,K_3)$ was based on Behrend’s construction of a large set of integers containing no 3-term
	arithmetic progression, % Alon and Shapira  considered sets of integers that do not contain subsets satisfying a certain set of equations. This structure, obtained via random methods, 
	Erdős, Frankl and Rödl \cite{EFR} proved, and later Alon and Shapira further extended \cite{Alon_Shapira} a generalization of Behrend's result for integer sets which avoid any triples from a $k$-term arithmetic progression. This result
	enables us to 
	construct a dense packing of cliques without a multicolor $K_3$. 
	
	Section 3 is devoted to the case of packing odd cycles without creating a multicolor $K_3$. Here we resolve the case when $F$ is a pentagon and provide lower and upper bounds for longer odd cycles as well. 
	%\todo{Még egy hangsúly arról, hgy ez miért random struktúrák és algoritmusok}
	
	\section{Proof of Theorem \ref{K_t K_3}}
	
	We will introduce the concept of a triple of integers with $h$-limited ratio, first appearing in a paper of Erdős, Frankl and Rödl \cite{EFR}. We present it in a more general framework introduced by Alon and Shapira \cite{Alon_Shapira}.
	
	\begin{defi}[$(k,h)$-gadget, \cite{Alon_Shapira}]
		Let $k\geq 3$, $h\geq 1$ be integers.
		Call a set of $k-2$ linear equations $\mathcal{E} = \{e_1,\ldots,e_{k-2}\}$ with integer coefficients in $k$ unknowns $x_1,\ldots,x_k$ a $(k,h)$-gadget if it satisfies the following properties:
		\begin{itemize}
			\item  Each of the unknowns $x_1,\ldots,x_k$ appears in at least one of the equations.
			\item For $1\le t\le k-2$, equation $e_t$
			is of the form
			$p_tx_i +q_tx_j = (p_t +q_t)x_{\ell},$
			where $0 < p_t,q_t \leq h$, and $x_i,x_j,x_{\ell}$ are distinct.
			\item Equations $e_1,\ldots,e_{k-2}$ are linearly independent.
		\end{itemize}
		We say that $z_1,\ldots,z_k$ satisfy a $(k,h)$-gadget $\mathcal{E}$ if they satisfy the $k-2$ equations of $\mathcal{E}$. Note that any
		gadget $\mathcal{E}$ has a trivial solution $x_1 = \ldots = x_k$.
	\end{defi}
	
	\begin{defi}[$(k,h)$-gadget-free, \cite{Alon_Shapira}]
		A set of integers $Z$ is called $(k,h)$-gadget-free if there are no $k$
		distinct integers $z_1,\ldots,z_k \in Z$ that satisfy an arbitrary $(k,h)$-gadget.
	\end{defi}
	Their main theorem is as follows.
	
	\begin{theorem}[Alon, Shapira, \cite{Alon_Shapira}]\label{ASS}
		For every $h$ and $k$ there is an integer $c = c(k,h)$, such that for every $n$ there is a $(k+1,h)$-gadget-free subset $Z \subset   \{1,2,\ldots,n\}$ of size at least
		$$|Z| \geq n\cdot \exp\left(-c(\log{n})^{\frac{1}{\lfloor\log{2k}\rfloor}}\right),$$ where logarithms are taken in base $2$.
	\end{theorem}
	
	We consider only the subcase when $k=3$. Note that in this case, the first and third condition hold automatically, hence $(3,h)$-gadgets can also be defined as follows:
	
	\begin{defi}[$h$-limited ratio]
		We say that a triple of integers $(a,b,c)$ is a triple with $h$-limited ratio if they satisfy an equation of the form $\frac{\lambda a+\mu b}{\lambda+\mu}=c$ for some positive integers $\lambda, \mu\le h$.
	\end{defi}
	
	Note that the case $h=1$ describes precisely the arithmetic progressions of length $3$. As a corollary of Theorem \ref{ASS}, we get back the theorem of Erdős, Frankl and Rödl \cite{EFR}.
	
	\begin{theorem}[]\label{AS} Fix a positive integer $h$.
		It is possible to choose a subset $A$ of $\{1,2,\dots,n\}$ of size  $n\cdot e^{-O(\sqrt{\log n})}$ which does not contain a triple with $h$-limited ratio. 
	\end{theorem}
	
	This result has been applied also by Gowers and Janzer in a closely related problem \cite{Gowers-Janzer} which also generalises the Ruzsa-Szemerédi theorem.   Let $r<t$ be positive integers and $G$ be a graph on $n$ vertices such
	that any of its subgraphs isomorphic to $K_r$ is contained in at most one subgraph isomorphic to $K_t$. Gowers and Janzer studied the largest number of copies of $K_t$ that $G$ can contain. Note that when $r=2$ and $t\geq 3$, if we take edge-disjoint copies of $K_t$ which avoid a multicolor $K_3$, then there will not be any copy of $K_t$ different from the original ones, so each copy of $K_2$ will be contained in exactly one copy of $K_t$.%Note that when $r=2$ and $t\ge 3$, if we take edge-disjoint copies of $K_t$ which avoid a multicolor $K_3$, then there will not be any copy of $K_t$ different from the original ones.\\
	
	In order to construct a large set of edge-disjoint copies of $K_t$ which avoid a multicolor $K_3$, will apply Theorem \ref{AS} with $h=t-2$.
	Take such a subset $A$,  and consider the following $t$-partite graph on $|V|=\frac{t(t+1)}{2}n$ vertices.
	
	%can be used to prove Theorem \ref{K_t K_3}.
	
	\begin{construction}\label{main_const}
		Let $V=\{(i,\ell): 1\le i\le t, 1\le \ell\le in\}$. Consider the set $\mathcal{S}$ of all $t$-element subsets of $V$ defined as $S_{\ell,k}=\{(1,\ell), (2,\ell+k), (3,\ell+2k), \dots, (t,\ell+(t-1)k)\}$ where $1\le \ell\le n$ and $k\in A$. For each element of $\mathcal{S}$, we take a copy of $K_t$ induced  by the $t$ corresponding vertices. Let $E$ be the union of the edge sets of these complete graphs $K_t$. 
	\end{construction}
	
	\begin{prop} 
		Construction \ref{main_const} defines a simple graph $(V,E)$ with a $K_t$ decomposition which
		\begin{itemize}
			\item does not contain a multicolor $K_3$
			\item  contains $ n^2e^{-O(\sqrt{\log n})}$ distinct copies of $K_t$.
		\end{itemize}
	\end{prop}
	\begin{proof}
		
		First observe that for any two vertices $v\ne v'\in V$, there is at most one pair of parameters $\ell$ and $k$ such that $v,v'\in S_{\ell,k}$, hence the edge sets of our copies of $K_t$ are indeed disjoint.\\
		Suppose to the contrary that the graph does contain a multicolor $K_3$. If $(i,x)$, $(j,y)$ and $(k,z)$ are the vertices of a multicolor triangle where $i<j<k$, then $a:=\frac{y-x}{j-i}$, $b:=\frac{z-y}{k-j}$ and $c:=\frac{z-x}{k-i}$ must be all distinct. Indeed,  if two of these 'slopes' were the same then actually the three vertices would all be in the same $K_t$.
		Hence these are three distinct elements of $A$ satisfying $(j-i)a+(k-j)b=(k-i)c$, where $1\le j-i,k-j\le t-2$, which is a contradiction to $A$ being  free of triples with $(t-2)$-limited ratio. \\ Finally note that as we have a $t$-set $S_{\ell, k}$ for each $\ell \in [1,n]$ and $k\in A$, the graph contains $ n^2e^{-O(\sqrt{\log n})}$ distinct copies of $K_t$ according to Theorem \ref{AS}.
	\end{proof}

	\section{Odd cycles, Proof of Theorem \ref{penta}}
	We start this section with the proof of Theorem \ref{penta} and prove that the maximum number of disjoint pentagons is $\left(\frac{1}{25}+o(1)\right)n^2$ in a multicolor triangle free graph.
	
	Consider the graph $H$ on $n$ vertices, which is obtained as the union of edge-disjoint pentagons $C_5^{(1)}, \ldots, C_5^{(t)}$, and suppose that $H$ contains  no multicolor triangles. Let $d(v)$ denote the degree of vertex $v$ in $H$. 
	Take the double sum $$\sum_{i=1}^t \sum_{v\in V(C_5^{(i)})} d(v).$$
	
	On the one hand, it is easy to see that the contribution of every vertex $v$ is exactly $\frac{d(v)^2}{2}$ since the edges incident to $v$ are distributed between $d(v)/2$ pentagons. Hence $$\sum_{i=1}^t \sum_{v\in V(C_5^{(i)})} d(v)= \frac{1}{2}\sum_{v\in V(H)} d^2(v)\geq \frac{n}{2}\left(\frac{2\cdot |E(H)|}{n}\right)^2= \frac{50t^2}{n}.$$
	by the QM-AM inequality.\\
	
	On the other hand,  $N(u)\cap N(v)=\emptyset$ holds for an edge $uv\in E(C_5^{(i)})$, unless there exists a vertex
	\begin{enumerate}[(a)]
		\item $w\in N(u)\cap N(v)$ such that $uw\in  E(C_5^{(i)})$ or $wv\in  E(C_5^{(i)})$, or
		\item $z\in N(u)\cap N(v)$ such that $uz\in  E(C_5^{(j)})$ and $zv\in  E(C_5^{(j)})$ for some $j\neq i$.
	\end{enumerate}
	
	Indeed, $H$ does not contain a multicolor triangle thus if $N(u)$ and $ N(v)$ are not disjoint, then their common neighbour, together with $u$ and $v$, must span 2 edges of the same color. The number of common neighbours $w$ of type (a) is at most $2$. The number of common neighbours $z$ of type (b) will be denoted by $N^*(uv)$. 
	
	\begin{lemma}\label{count}For every $1\leq i \leq t$,
		$$\sum_{v\in V(C_5^{(i)})} d(v)\le 2n + 10+ \sum_{uv\in E(C_5^{(i)})} N^*(uv).$$
	\end{lemma}
	
	\begin{proof}
		Let us index the vertices along the cycle $C_5^{(i)}$, i.e., $V(C_5^{(i)})=\{v_1, v_2, v_3, v_4, v_5\}$. 
		
		The left hand side counts the edges $uw$ which have an endvertex $v$ from $V(C_5^{(i)})$. If an edge has both endpoints in this set, it is counted twice.
		
		%The left hand side corresponds to the number of ordered pairs $(v,x)$ where $v\in V$, $x\in V(H)$ and $vx\in E(H)$. Let us count the number of such pairs by iterating through the vertices $x$: $\sum\limits_{v\in V} d(v)=\sum\limits_{x\in V(H)} d_V(x)$ where $d_V(x)=|N(x)\cap V|$.
		
		We distinguish two cases.\\ If $w\in V(C_5^{(i)})$, then $w$ contributes $2$ to the sum  for its two neighbours on the cycle, and contributes  at most $2$ for its non-neighbours on the cycle. \\
		Suppose that $w\not\in V(C_5^{(i)})$. Then $w$ cannot be joined to consecutive vertices along $C_5^{(i)}$, unless the edges connecting $w$ and the consecutive vertices $v_{\ell}, v_{\ell+1}$ are of the same color, i.e., coming from the same copy $C_5^{(j)}$.  (Here $v_6\equiv v_1$.) If we label $v_{\ell+1}$ in each such consecutive pair which is connected to $w$ with the same color, then for each $w$ its unlabelled neighbours form an independent set within the pentagon $C_5^{(i)}$, otherwise a multicolor triangle would be formed. Hence the cardinality of the unlabelled neighbours of $w$ is most $2$. On the other hand, the labelled neighbours are in one-to-one correspondence with ${\sum_{uv\in E(C_5^{(i)})} N^*(uv).}$ All in all, we get at most $5\cdot 4+(n-5)\cdot 2 +{\sum_{uv\in E(C_5^{(i)})} N^*(uv).}$ The bound thus follows.\end{proof}

	We apply Lemma \ref{count} to obtain 
	
	$$\sum_{i=1}^t \sum_{v\in V(C_5^{(i)})} d(v)\leq (2n+10)t+\sum_{i=1}^t\sum_{uv\in E(C_5^{(i)})} N^*(uv).$$
	
	To bound $\displaystyle{\sum_{i=1}^t\sum_{uv\in E(C_5^{(i)})} N^*(uv)},$ observe that when we count the common neighbours of type (b), each common neighbour is a vertex $z$ of a pentagon $C_5^{(j)}$ for some $j\in \{1, \ldots, t\}$ and every $z$ appears at most once, namely for a possible edge $uw$ where $u$ and $w$ are neighbours of $z$ in  $C_5^{(j)}$. Consequently $$\sum_{i=1}^t\sum_{uv\in E(C_5^{(i)})} N^*(uv)\leq 5t.$$
	Putting this all together, we get 
	$$ \frac{50t^2}{n}\le \sum_{i=1}^t \sum_{v\in V(C_5^{(i)})} d(v)\leq  (2n+15)t,$$
	which in turn implies 
	$$t\le \frac{n^2}{25}+O(n). \eqno(*)$$ \qed
	
	\begin{remark}
		Note that a linear term is indeed needed in $(*)$, as for $n=5$ one can decompose $K_5$ into two edge-disjoint pentagons, so $C_5[1]$ does not provide an extremal construction for $\ex_{C_5}(n=5, K_3)$.
	\end{remark}
	
	Observe that $C_5[n/5]$ is the extremal construction for the maximum number of pentagons in an $n$-vertex triangle-free graph, as conjectured by Erdős in 1984 \cite{Erdos84} and resolved by Grzesik \cite{Grzesik} and independently by Hatami, Hladk\'y, Král', Norine and Razborov \cite{Hatami}. This problem can be stated as a generalised Turán problem: $\ex(n,H,G)$  denotes the maximum number of copies of $H$ in $G$-free graphs on $n$ vertices.  Recently, Grzesik and Kielak \cite{G+K} investigated $\ex(n,C_k,C_l)$ for odd integers $k>l$, and conjectured that  $\ex(n,C_k,C_l)$ is asymptotically attained at the balanced blow-up of an $(l + 2)$-cycle, moreover they proved it for $k=l+2\geq 3$ as well. %In fact this is very closely related to the problem  of finding the maximum number of induced cycles of a given length \cite{Pip}.
	
	\begin{proof}[Proof of Theorem \ref{oddlarge}, upper bound]
		Following the same lines as the proof of Theorem \ref{penta}, on the one hand we get
		
		$$   \frac{2(2k+1)^2t^2}{n} \le \sum_{i=1}^t \sum_{v\in V(C_5^{(i)})} d(v),
		$$
		whereas on the other hand, the analogue of Lemma \ref{count} gives
		$$\sum_{v\in V(C_{2k+1}^{(i)})} d(v)\le n\alpha(C_{2k+1})+2|V(C_{2k+1}^{(i)})|+\sum_{uv\in E(C_{2k+1}^{(i)})} N^{*}(uv)$$
		(for each $1\le i\le t$) which means
		$$\sum_{v\in V(C_{2k+1}^{(i)})} d(v)\le kn+2(2k+1)+\sum_{uv\in E(C_{2k+1}^{(i)})} N^{*}(uv).$$
		Since we have
		$$\sum_{i=1}^t \sum_{uv\in E(C_{2k+1}^{(i)})} N^{*}(uv)\le (2k+1)t,$$
		altogether we get 
		$$\frac{2(2k+1)^2t^2}{n}\le (kn+3(2k+1))t$$
		giving
		$$t\le \frac{k}{2(2k+1)^2}n^2+O(n).$$
		Therefore
		\begin{equation*}\ex_{C_{2k+1}}(n,K_3)\leq (1+o(1))\frac{k}{2(2k+1)^2} n^2,\end{equation*}  confirming the upper bound for the second statement in Theorem \ref{oddlarge}.
	\end{proof} 
	
	As we have seen before, the $C_5$-decomposition of $C_5[n/5]$ provides an asymptotically extremal construction for $\ex_{C_{2k+1}}(n,K_3)$ when $k=2$.
	For the general problem $\ex_{C_{2k+1}}(n,K_3)$ $k\geq 2$, an analogous construction would be to decompose the blow-up of $C_{2k+1}[n/(2k+1)]$; however this blow-up would be a much sparser graph than the blow-up of $C_{5}[n/5]$. Thus, 
	it would be natural to consider an asymptotic decomposition of $C_5[n/5]$, if it exists, to cycles $C_{2k+1}$ for $k>2$ for the lower bound on $\ex_{C_{2k+1}}(n,K_3)$.
	%\cdot |E(C_{2k+1})|>(1+o(1))\frac{n^2}{5}$,
	This would be also in the spirit of the conjecture of Grzesik and Kielak concerning the generalised Turán number of odd cycle pairs. That decomposition indeed exists as we will see below, and provides $(1+o(1))\frac{1}{5(2k+1)} n^2$ $(2k+1)$-cycles, but it is  below the stated lower bound $(1+o(1))\frac{k}{2(2k+1)^2} n^2.$ As detailed below, we achieve a better construction by breaking the symmetry on the class sizes of the blow-up.\\
	
	We continue by proving Theorem \ref{oddlarge}, and  we introduce some necessary lemmas first. Let us recall the Haxell-Rödl theorem. In order to state it, we will have to introduce some notions. 
	Let $F$ and $G$ be two graphs. We denote by ${G \choose H}$ the set of all subgraphs of $G$ isomorphic to $H$. We call a subset $\mathcal{A}$ of ${G \choose H}$ an $H$-packing if any two graphs $A,B \in \mathcal{A}$ are edge-disjoint. Let $\nu_H(G)$ be the size of the largest $H$-packing. A fractional $H$-packing is a function $\psi: {G \choose H} \to [0,1]$ such that for each edge $e \in E(G)$ we have $\sum_{e \in H' \in {G \choose H}} \psi(H') \leq 1$. Let $|\psi|=\sum_{H' \in {G \choose H}} \psi(H')$. We say that a fractional $H$-packing is maximal if $|\psi|$ is maximal and  denote this maximum by $\nu_H^*(G).$
	
	\begin{theorem}[Haxell--Rödl] \label{HR}
		Let $H$ be a fixed graph and let an $\eta>0$ be given. Then there exists an integer $N$ such that for any graph $G$ on $n \geq N$ vertices we have
		$$\nu_H^*(G)-\nu_H(G) \leq \eta n^2.$$
	\end{theorem}

	%\todo{a Grzesik-féle gene Turán problem for odd cycles esettel összevetni. sejtéseket megfogalmazni. Beírni még az introba a Rödl Haxellt és alkalmazni amivel látszódik az alsó korlát nagyobb ptn körök esetére}
	
	\begin{comment}
	%%%%%%%%%%%%%%%%%%%%%%%%%%%%%%%%%%%%%%%%%%%%%%%%%%%%%%
	\begin{lemma}\label{tortkonst} Let $k\geq3$ an integer and fix some $\alpha, \beta \in \mathbb{Q}$ such that $1/2 \geq \alpha \geq \beta\geq 0$. Suppose that $H_{\alpha, \beta}$ is a graph on $n$ vertices such that $V(H_{\alpha, \beta})=A_1\dot\cup A_2\dot\cup C$, $B_i\subseteq A_i$ for $i\in \{1,2\}$, with $|B_i|=\beta\cdot n$, $|A_i|=\alpha\cdot n$, $|C|=(1-2\alpha)n$ and the edges are joining every pair of vertices from $B_i$ and $C$, $A_1\setminus B_1$ and $A_2$, $A_2\setminus B_2$ and $A_1$, see Figure \ref{fig}.
	If  $$ (i) \ \ \frac{\alpha^2-\beta^2}{2\beta(1-2\alpha)}=\frac{2k-1}{2},$$ or 
	$$ (ii) \ \ \ \alpha=2\beta=0.2,$$ then
	there is a perfect fractional $C_{2k+1}$ packing of $H_{\alpha, \beta}$, i.e., 
	$\nu_H^*(G)=\frac{|E(H_{\alpha, \beta})| }{2k+1}.$
	\end{lemma}
	Note that $H_{0.4, 0.2}=C_5[n/5]$.

	\begin{prop} Let $k>2$ integer.
	Suppose that $$\beta=\frac{-(2k-1)(1-2\alpha)+\sqrt{(2k-1)^2(1-2\alpha)^2+4\alpha^2}}{2},$$
	
	then we can choose $\alpha\in [0,0.5]$ such that 
	\begin{itemize}
	\item $|E(H_{\alpha, \beta})|>|E(H_{0.4, 0.2})|$,
	\item $\lim_{k\to \infty}\frac{|E(H_{\alpha, \beta})|}{n^2} \to \frac{1}{4}.$
	\end{itemize}
	\end{prop}
	
	\end{comment}
	%%%%%%%%%%%%%%idáig kommentelve
	
	\begin{nota}
		For two disjoint sets $X$ and $Y$, $E[X,Y]$ refers to a complete bipartite graph spanned by the vertex classes $X$ and $Y$.
	\end{nota}
	
	\begin{construction}[Unbalanced blow-ups]
		Let $\alpha, \beta, \gamma \in \mathbb{Q}^+$ such that $2\alpha+2\beta+\gamma=1$. Define the unbalanced blow-up graph $G_{\alpha, \beta, \gamma}(n)$ of $C_5$, on $n$ vertices, with parameters $\alpha, \beta, \gamma$ such that 
		\begin{itemize}
			\item $V(G_{\alpha, \beta, \gamma}(n))=A_1\dot\cup A_2\dot\cup B_1 \dot\cup B_2\dot\cup C,$
			\item $\displaystyle{E(G_{\alpha, \beta, \gamma}(n))=E[A_1,A_2]\cup E[A_1,B_1]\cup E[B_1,C]\cup E[B_2,C]\cup E[A_2, B_2]},$
			\item $|A_i|=\alpha\cdot n$, $|B_i|=\beta\cdot n$, $|C|=\gamma\cdot n$.
		\end{itemize}
	\end{construction}
	
	\noindent Note that $G_{0.2, 0.2, 0.2}(n)$ is simply $C_5[n/5].$\\
	\begin{proof}[Proof of Theorem \ref{oddlarge}, lower bound]
		
		Fix an integer $k\geq 3$. Let $\mathcal{C}_A$ be the family of all $C_{2k+1}$ subgraphs of $G_{\alpha, \beta, \gamma}$ which have exactly $2k-3$ edges in $E[A_1,A_2]$.
		Give each element $C_{2k+1}\in \mathcal{C}_A$ a uniform weight
		$$\psi(C_{2k+1})=\frac{\delta}{|\mathcal{C}_A|}.$$
		Similarly, let $\mathcal{C}_{A_iB_i}$ be the family of all $C_{2k+1}$ subgraphs of $G_{\alpha, \beta, \gamma}$ which have exactly $2k-3$ edges in $E[A_i,B_i]$ for some $i\in \{1,2\}$.
		Give each element $C_{2k+1}\in \mathcal{C}_{A_iB_i}$ a uniform weight 
		$$\psi(C_{2k+1})=\frac{\mu}{|\mathcal{C}_{A_iB_i}|}.$$ 
		Finally, let $\mathcal{C}_{B_iC}$ be the family of all $C_{2k+1}$ subgraphs of $G_{\alpha, \beta, \gamma}$ which have exactly $2k-3$ edges in $E[B_i,C]$ for some $i \in \{1,2\}$.
		Give each element $C_{2k+1}\in \mathcal{C}_{B_iC}$ a uniform weight
		$$\psi(C_{2k+1})=\frac{\lambda}{|\mathcal{C}_{B_iC}|}.$$
		
		We give zero weight to the other cycles $C_{2k+1}$. We would like to choose $\delta, \mu, \lambda$ so that this fractional packing becomes perfect, that is, the weights of the cycles containing each edge sum to $1$. To achieve this, we must have 
		\begin{equation}
			\begin{aligned}
				|A_1||A_2|&=\alpha^2 \cdot n^2=(2k-3)\delta+\mu+\lambda,\ \ \ \\
				2|A_i||B_i|&=2\alpha\beta\cdot n^2=2\delta+(1+(2k-3))\mu+2\lambda, \ \ \ (i \in \{1,2\})\\
				2|B_i||C|&=2\beta\gamma\cdot n^2=2\delta+2\mu+(1+(2k-3))\lambda. \ \ \ (i \in \{1,2\})\\
			\end{aligned}  \label{eq:pair-dense} 
		\end{equation}
		
		Indeed, any cycle $C_{2k+1}$ containing exactly $2k-3$ edges between two neighbouring classes will contain exactly one edge between any other pair of neighbouring classes. Observe that the cardinality of the classes of $V(G_{\alpha, \beta, \gamma}(n))$ can be expressed from (\ref{eq:pair-dense}) as follows.
		
		\begin{equation}
			\begin{aligned}
				|B_i|/|A_i|&=\beta/\alpha=\frac{\lambda+(k-1)\mu+\delta}{\lambda+\mu+(2k-3)\delta}, \ \ \  (i\in \{1,2\})\\
				|C|/|A_i|&=\gamma/\alpha=\frac{(k-1)\lambda+\mu+\delta}{\lambda+(k-1)\mu+\delta}, \ \ \  (i\in \{1,2\})
			\end{aligned}  \label{eq:cardi} 
		\end{equation}
		
		Together from (\ref{eq:pair-dense}) and (\ref{eq:cardi}), we can express the edge density of an unbalanced blow-up of a $C_5$ which has a perfect fractional $C_{2k+1}$-packing. Our aim is to give a lower bound
		on the edge density which is better than that of $C_5[n/5]$.
		
		\begin{equation}
			\begin{aligned}
				\frac{|E(G_{\alpha, \beta, \gamma}(n))|}{n^2}&=\frac{|A_1||A_2|+ 2|A_1||B_1|+2|B_1||C|}{(2|A_1|+2|B_1|+|C|)^2}\\
				&=\frac{(2k+1)(\lambda+\mu+\delta)}{|A_1||A_2|(2+2\frac{\lambda+(k-1)\mu+\delta}{\lambda+\mu+(2k-3)\delta}+ \frac{(k-1)\lambda+\mu+\delta}{\lambda+(k-1)\mu+\delta})^2}\\
				&= \frac{(2k+1)(\lambda+\mu+\delta)}{(\lambda+\mu+(2k-3)\delta)(2+2\frac{\lambda+(k-1)\mu+\delta}{\lambda+\mu+(2k-3)\delta}+ \frac{(k-1)\lambda+\mu+\delta}{\lambda+(k-1)\mu+\delta})^2}.
			\end{aligned}  \label{eq:density} 
		\end{equation}
		
		From this point, a simple maximization problem on $\lambda, \mu, \delta$ derives the largest edge density which can be attained asymptotically via this approach. This would lead to a rather complex formula using roots of cubic polynomials thus we decide to pick a triple of values for $(\lambda, \mu, \delta)$ for each $k\geq 3$ as  
		$$(\lambda, \mu, \delta)=\left(0,1,-0.5+\frac{\sqrt[3]{4k+15}}{2}\right) \mbox{ \ \ for } k.$$%\todo{átszámolni, vmi ne kóser}
		
		Easy calculation shows that this provides %maximize (  101(x+y+z)/(x+y+97z)/(2+2(x+49y+z)/(x+y+97z)+(49x+y+z)/(x+49y+z))^2  ) , 0<=x, 0<=y, 0<=z
		
		$$ \frac{|E(G_{\alpha, \beta, \gamma}(n))|}{n^2}=0.2016>0.2= \frac{|E(G_{0.2, 0.2, 0.2}(n))|}{n^2} \mbox{ \ \ for } k=3;$$  and in general if we determine $\alpha, \beta, \gamma$ via (\ref{eq:cardi}) with $(\lambda, \mu, \delta)=(0,1,\frac{-1+\sqrt[3]{4k+15}}{2}) $, we obtain that $ \frac{|E(G_{\alpha, \beta, \gamma}(n))|}{n^2}$ is monotone increasing in $k\geq 3$ and

		$$ \lim_{k\to \infty} \frac{|E(G_{\alpha, \beta, \gamma}(n))|}{n^2}=\frac{k\sqrt[3]{4k}}{k\sqrt[3]{4k}\cdot 4}.$$
		
		This verifies Theorem \ref{oddlarge}.
	\end{proof}
	\begin{remark}
		Exact calculation shows that the solution of the maximization of expression (\ref{eq:density}) is $$-(7 (-350 - 29093/(5865445 + 170859 \sqrt{2022})^{1/3} + (5865445 + 170859 \sqrt{2022})^{1/3}))/8112\approx 0.201615$$ for $k=3,$ i.e., in the case of packing $C_{7}$ graphs without creating a multicolor $K_3$.
	\end{remark} %\todo{C3 vs C5 közben megvan: https://www.sciencedirect.com/sdfe/reader/pii/S0012365X2200334X/pdf }

	\noindent {\bf Acknowledgement}\\
	We would like to thank the anonymous referees for their valuable suggestions in their thorough report which helped us in improving the paper.
	
	%Before we uploaded the paper to the ArXiv, Letícia Mattos pointed out  that together with  József Balogh, Anita Liebenau and Natasha Morrison, they also  proved Theorem 1.6 with a different method, see \cite{Mattos}.\\
	
	%\todo{Gerbner cikkre esetleg lehetne hivatkozni}

	%\section{Concluding remarks}

\end{document}